\documentclass[11pt, reqno, psamsfonts]{amsart}
\pdfoutput=1

\usepackage{amssymb}
\usepackage{amsthm}
\usepackage{amsmath}
\usepackage{latexsym}

\usepackage[T1]{fontenc}
\usepackage[utf8]{inputenc}
\usepackage[russian, french, english]{babel}

\usepackage{graphicx}
\usepackage{wrapfig}
\usepackage[justification=centering, labelfont=bf]{caption}
\usepackage{mathtools}
\usepackage{amsbsy}
\usepackage[inline]{enumitem}
\usepackage{mathrsfs}
\usepackage{array}
\usepackage{multicol}
\usepackage{stmaryrd}
\usepackage{cancel}
\usepackage{lmodern}
\usepackage{mathabx}
\usepackage{upgreek}
\usepackage{titlesec}
\usepackage{titletoc}
\usepackage[spacing=true,kerning=true,babel=true,tracking=true]{microtype}
\usepackage[shortcuts]{extdash}
\usepackage[foot]{amsaddr}
\usepackage[left=1in,right=1in,top=1in,bottom=1in,bindingoffset=0cm]{geometry}
\usepackage{bm}
\usepackage{centernot}
\usepackage{mdframed}
\usepackage[hidelinks]{hyperref}
\usepackage{xspace}
\usepackage[skins]{tcolorbox}
\usepackage{framed}
\usepackage{tikz}
\usetikzlibrary{shapes,snakes}
\usetikzlibrary{arrows.meta}
\usetikzlibrary{decorations.pathmorphing}
\usetikzlibrary{patterns}
\usepackage{float}
\usepackage{cleveref}
\usepackage{csquotes}

\usepackage[
    sortcites,
    backend=biber, style=alphabetic, sorting=nyt, maxnames=100,backref=true]{biblatex}
\title{\sffamily Hypergraph independence bounds: from maximum degree to average degree} 
\date{}

\author{Jing~Yu}
\address{\normalfont (JY) Shanghai Center for Mathematical Sciences, Fudan University, Shanghai, China}
\email{jyu@fudan.edu.cn}

\author{Junchi~Zhang}
\address{\normalfont (JZ) Shanghai Center for Mathematical Sciences, Fudan University, Shanghai, China}
\email{jczhang24@m.fudan.edu.cn}

\thanks{JY's research is partially supported by the National Natural Science Foundation of China grant 12371343 and 12525110 (PI: Hehui Wu).}

\newtheoremstyle{bfnote}%
{}{}%
{\slshape}{}%
{\bfseries}{\bfseries.}%
{ }%
{\thmname{#1}\thmnumber{ #2}\thmnote{ \ep{\normalfont{}#3}}}

\theoremstyle{bfnote}
\newtheorem{theo}{Theorem}[section]
\newtheorem*{theo*}{Theorem}
\newtheorem{prop}[theo]{Proposition}
\newtheorem{lemma}[theo]{Lemma}

\newtheorem{corl}[theo]{Corollary}

\newtheorem*{corl*}{Corollary}

\theoremstyle{definition}
\newtheorem{defn}[theo]{Definition}
\newtheorem*{defn*}{Definition}

\newtheorem{remk}[theo]{Remark}

\newtheorem*{exmp*}{Example}

\theoremstyle{remark}
\newtheorem*{ques*}{Question}
\newtheorem*{remk*}{Remark}

\newcommand*{\myproofname}{Proof}

\makeatletter
\newcommand{\neutralize}[1]{\expandafter\let\csname c@#1\endcsname\count@}
\makeatother

\newcommand{\N}{{\mathbb{N}}}

\newcommand{\G}{\mathcal{G}}
\renewcommand{\epsilon}{\varepsilon}
\newcommand{\eps}{\epsilon}
\renewcommand{\phi}{\varphi}
\renewcommand{\theta}{\vartheta}
\renewcommand{\leq}{\leqslant}
\renewcommand{\geq}{\geqslant}
\newcommand{\defeq}{\coloneqq}

\newcommand{\bemph}[1]{{\normalfont#1}}
\newcommand{\ep}[1]{\bemph{(}#1\bemph{)}}

\newcommand{\emphd}[1]{{\fontseries{b}\selectfont\textsf{#1}}}

\numberwithin{equation}{section}

\titleformat{\section}[block]{\large\bfseries\sffamily}{\thesection.}{1ex}{}
\titleformat{\subsection}[block]{\bfseries\sffamily}{\thesubsection.}{1ex}{}
\titleformat{\subsubsection}[block]{\itshape}{\bfseries\upshape\sffamily\thesubsubsection.}{1ex}{}

\titlespacing*{\section}{0pt}{*3}{*1}
\titlespacing*{\subsection}{0pt}{*3}{*1}
\titlespacing*{\subsubsection}{0pt}{*2}{*1}

\titlecontents{section}
 [1.5em]
 {\smallskip}
 {\bfseries\thecontentslabel\hspace{1.02em}}
{\bfseries}
 {\,\,\titlerule*[0.77pc]{}\bfseries\contentspage}
\titlecontents{subsection}
 [4em]
 {\smallskip}
 {\thecontentslabel\hspace{1.02em}}
{\hspace*{2.32em}}
 {\,\,\titlerule*[0.77pc]{.}\contentspage}

\renewbibmacro{in:}{}

\renewbibmacro*{volume+number+eid}{%
	\printfield{volume}%
	\setunit*{\addnbspace}%
	\printfield{number}%
	\setunit{\addcomma\space}%
	\printfield{eid}}

\DeclareFieldFormat[article]{volume}{\textbf{#1}\space}
\DeclareFieldFormat[article]{number}{\mkbibparens{#1}}

\DeclareFieldFormat{journaltitle}{#1,}
\DeclareFieldFormat[thesis]{title}{\mkbibemph{#1}\addperiod}
\DeclareFieldFormat[article, unpublished, thesis]{title}{\mkbibemph{#1},}
\DeclareFieldFormat[book]{title}{\mkbibemph{#1}\addperiod}
\DeclareFieldFormat[unpublished]{howpublished}{#1, }

\DeclareFieldFormat{pages}{#1}
\DeclareFieldFormat[article]{series}{Ser.~#1\addcomma}

\setlist{topsep=3pt,itemsep=3pt}

\begin{document}

\begin{abstract}
   We prove a transfer theorem for hereditary classes of $(r+1)$-uniform hypergraphs. Let $\mathcal H$ be such a class, and for $H\in\mathcal H$ write $\Delta(H)$ and $d(H)$ for the maximum degree and average degree of $H$, respectively. We show that, for every nearly logarithmic function $f$ in the sense defined below, a maximum-degree lower bound for the independence number of the form
\[
\alpha(H)\ge (1-o(1))\frac{f(\Delta(H))}{\Delta(H)^{1/r}}|V(H)|
\qquad\text{as }\Delta(H)\to\infty
\]
for all $H\in\mathcal H$ implies the corresponding average-degree lower bound
\[
\alpha(H)\ge (1-o(1))\frac{f(d(H))}{d(H)^{1/r}}|V(H)|
\qquad\text{as }d(H)\to\infty .
\]
We combine this transfer theorem with known coloring and fractional-coloring bounds to obtain consequences for graphs excluding a fixed cycle, graphs with bounded clique number, locally $q$-colorable graphs, and locally sparse uniform hypergraphs. 
\end{abstract}
\maketitle

\section{Introduction}\label{sec:intro}

Let $H$ be an $(r+1)$-uniform hypergraph. We write $\Delta(H)$ and $d(H)$ for its maximum degree and average degree, respectively. The elementary random deletion argument gives
\[
\alpha(H)=\Omega\left(\frac{|V(H)|}{d(H)^{1/r}}\right),
\]
and the analogous maximum-degree bound with $d(H)$ replaced by $\Delta(H)$ is immediate.
A recurring theme in extremal graph theory is that additional local structure can improve these bounds, often by a logarithmic or polylogarithmic factor.

A classical example is the triangle-free case. Ajtai, Komlós, and Szemerédi~\cite{Ajtai1980Ramsy} proved that every triangle-free graph $G$ of maximum degree $\Delta$ has an independent set of size at least
\[
c\,\frac{\log\Delta}{\Delta}|V(G)|
\]
for an absolute constant $c>0$. Shearer~\cite{Shearer1983ind1} later proved the corresponding average-degree estimate
\[
\alpha(G)\ge (1-o_d(1))\frac{\log d}{d}|V(G)|
\]
for any triangle-free graph $G$ of average degree $d$.
The passage from maximum degree to average degree is not formal: a graph of small average degree may contain vertices of very large degree, and arguments designed for bounded maximum degree often do not apply directly.
Over the decades, there are plenty of results related to independence number under such constraints, such as \cite{Spencer1982hyper,Rodl1995hyper,alon1999coloring,Anton2023ForbidBipartite,Verstraete2026hyper}. Some of them are in terms of maximum degree and some are in terms of average degree.

The purpose of this note is to isolate a simple transfer principle that handles this difficulty for uniform hypergraphs.We shall use the term \emphd{hereditary} in the induced sense: a class of hypergraphs is hereditary if every induced subhypergraph of a member of the class again belongs to the class. For graphs, this is the usual notion of being closed under taking induced subgraphs. We say that a graph or hypergraph $H$ is $F$-free if it contains no subgraph isomorphic to $F$; unless explicitly stated otherwise, subgraphs are not required to be induced.

Suppose that, on a hereditary class of $(r+1)$-uniform hypergraphs, one knows a maximum-degree lower bound of the form
\[
\alpha(H)\ge (1-o_\Delta(1))\frac{f(\Delta)}{\Delta^{1/r}}|V(H)|.
\]
Under a mild slow-variation assumption on $f$, we show that the corresponding average-degree bound
\[
\alpha(H)\ge (1-o_d(1))\frac{f(d)}{d^{1/r}}|V(H)|
\]
follows automatically.

The proof is a cleaning argument. Starting from $H$, we repeatedly delete a vertex whose degree is larger than a fixed multiple of the current average degree. The quantity
\[
\frac{|V(H_i)|}{d(H_i)^{1/r}}
\]
increases during this process. Thus either the average degree drops enough that a crude independent-set bound suffices, or many vertices have been deleted and this quantity has grown enough, or the process stops at an induced subhypergraph whose maximum degree is comparable to its average degree. In the last case the assumed maximum-degree theorem applies. Our cleaning argument is inspired by the recent work of Dhawan, Janzer, and Methuku \cite{Abhi2025Kttt} on $K_{t,t,t}$-free graphs, where a novel high-degree deletion procedure is combined with a Rödl nibble to obtain sharp average-degree independence bounds. Here we isolate the deterministic part of this idea and formulate it as a general maximum-degree-to-average-degree transfer theorem for hereditary uniform hypergraph classes.

We then combine the transfer theorem with known maximum-degree coloring and independence bounds to obtain average-degree consequences for several graph and hypergraph classes. The proofs of these consequences are given in Section \ref{sec:apps}; the proof of the transfer theorem is given in Section \ref{sec:proof}.

To state the transfer theorem, we isolate the following class of functions.

\begin{defn}\label{def:nearly-log}
A function $f\colon (0, \infty) \to (0, \infty)$ is \emphd{nearly logarithmic} if the following hold: 
\begin{enumerate}[label=\ep{\normalfont{}C\arabic*}]
    \item \label{item:C1} $f(x)=x^{o(1)}$ as $x\to +\infty \quad \text{and}\quad \liminf_{x\to \infty}f(x)>2.$
    \item \label{item:C2} For any $\eps>0$ and $m\geq 0$, there exist $M>0$ and $\lambda>1$ such that for every $x>M$,
    \[\inf_{y\in[xf(x)^{-m},\lambda x]}\frac{f(y)}{f(x)}\ge 1-\varepsilon.\]
\end{enumerate}
\end{defn}

With this definition, our main theorem is stated as follows.

\begin{theo}\label{thm:main}
    Fix $r\geq 1$, let $f$ be a nearly logarithmic function, and let $\G$ be a hereditary class of $(r+1)$-uniform hypergraphs. Suppose that for every $\delta > 0$ there exists $\Delta_0$ such that every hypergraph $H\in\G$ with $\Delta(H) \ge \Delta_0$ satisfies
    $$ \alpha(H)\geq (1-\delta)  \dfrac{f(\Delta(H))}{\Delta(H)^{1/r}} |V(H)|.$$
    Then for every $\epsilon > 0$ there exists $d_0$ such that every $H\in \G$ with $d(H) \ge d_0$ satisfies
    $$ \alpha(H)\geq (1-\epsilon)  \dfrac{f(d(H))}{d(H)^{1/r}} |V(H)|.$$
\end{theo}


\section{Nearly logarithmic functions and applications}\label{sec:apps}

This section contains the ingredients needed to apply Theorem~\ref{thm:main}. We first give a simple criterion for checking that a function is nearly logarithmic. We then combine the transfer theorem with known maximum-degree and degeneracy bounds to obtain the stated applications.

\subsection{A sufficient condition}

The definition of a nearly logarithmic function is tailored to the cleaning argument, but in the applications it is usually verified through a simpler differential condition.

\begin{prop}\label{prop:easy-nearly-log}
Let $f:(0,\infty)\to (0, \infty)$ be eventually increasing and continuously differentiable.
Assume that
\[
f(x)=x^{o(1)},\qquad \liminf_{x\to\infty}f(x)>2,
\]
and
\[
\lim_{x \to \infty} \frac{x f'(x)}{f(x)}\log f(x) = 0.
\]
Then $f$ is nearly logarithmic.
\end{prop}

\begin{proof}
Condition \ref{item:C1} is immediate from the assumptions.

We verify condition \ref{item:C2}. Fix $\varepsilon>0$ and $m\ge0$. If $m=0$, the conclusion follows immediately from eventual monotonicity by taking any fixed $\lambda>1$. Thus assume $m>0$. We may also assume $0<\varepsilon<1$. Choose $\delta>0$ so small that
\[
e^{-2m\delta}\ge 1-\varepsilon.
\]
Since $f$ is eventually increasing, there exists $x_0$ such that $f$ is increasing on
$[x_0,\infty)$. Set
\[
u(x):=\log f(x).
\]
Since $\liminf_{x\to\infty}f(x)>2$, after increasing $x_0$ if necessary, there exists
$\gamma>0$ such that
\[
u(x)\ge \gamma
\qquad\text{for all }x\ge x_0.
\]
Moreover, as $f$ is increasing and continuously differentiable on $[x_0,\infty)$, we have
\[
u'(x)\ge0
\qquad\text{for all }x\ge x_0.
\]

For each sufficiently large $x$, define
\[
y_x:=xf(x)^{-m}=xe^{-m u(x)}.
\]
Since $f(x)=x^{o(1)}$, we have
\[
u(x)=\log f(x)=o(\log x).
\]
Therefore
\[
\log y_x=\log x-mu(x)=\log x-o(\log x)\to\infty.
\]
In particular, $y_x\to\infty$.

Let
\[
A(t):=t u'(t)u(t)
=
\frac{t f'(t)}{f(t)}\log f(t).
\]
By assumption, $A(t)\to0$. Hence, after increasing $x$ if necessary, we have
\[
A(t)\le \delta
\qquad\text{for all }t\in[y_x,x].
\]
For such $x$, we estimate
\[
u(x)^2-u(y_x)^2
=
\int_{y_x}^{x}2u(t)u'(t)\,\mathrm{d}t
=
\int_{y_x}^{x}\frac{2A(t)}{t}\,\mathrm{d}t
\le
2\delta\log\frac{x}{y_x}.
\]
By the definition of $y_x$,
\[
\log\frac{x}{y_x}=m u(x).
\]
Thus
\[
u(x)^2-u(y_x)^2\le 2m\delta u(x).
\]
Since $u$ is increasing on $[x_0,\infty)$ and $y_x\to\infty$, for all sufficiently large $x$ we have
$u(y_x)\le u(x)$ and hence
\[
u(x)-u(y_x)
=
\frac{u(x)^2-u(y_x)^2}{u(x)+u(y_x)}
\le
\frac{2m\delta u(x)}{u(x)}
=
2m\delta.
\]
It follows that for every $y\in[y_x,x]$,
\[
0\le u(x)-u(y)\le u(x)-u(y_x)\le 2m\delta.
\]
Consequently,
\[
\frac{f(y)}{f(x)}
=
\exp(u(y)-u(x))
\ge e^{-2m\delta}
\ge 1-\varepsilon
\tag{2.1}\label{eq:left}
\]
for every $y\in[xf(x)^{-m},x]$ and all sufficiently large $x$.

Now choose, for instance, $\lambda=2$. Since $f$ is increasing on $[x_0,\infty)$, for every
$y\in[x,\lambda x]$ and all sufficiently large $x$ we have
\[
\frac{f(y)}{f(x)}\ge 1.
\tag{2.2}\label{eq:right}
\]
Combining \eqref{eq:left} and \eqref{eq:right}, we obtain
\[
\frac{f(y)}{f(x)}\ge 1-\varepsilon
\]
for every $y\in[xf(x)^{-m},\lambda x]$ and all sufficiently large $x$. This proves condition \ref{item:C2}.
\end{proof}

\begin{remk}
Proposition~\ref{prop:easy-nearly-log} applies to
\[
\log x,\qquad
(\log x)^\beta,\qquad
\frac{\log x}{\log\log x},\qquad
\log\log x,
\]
after modifying these functions arbitrarily on a bounded interval to make them positive on
$(0,\infty)$. It also applies to
\[
\exp((\log x)^\alpha),\qquad 0<\alpha<\frac12.
\]
Indeed, for $f(x)=\exp((\log x)^\alpha)$ one has
\[
\frac{x f'(x)}{f(x)}\log f(x)
=
\alpha(\log x)^{2\alpha-1}\to0
\]
precisely when $\alpha<1/2$.
\end{remk}

We now turn to applications. In each case, an existing coloring or fractional-coloring theorem gives a maximum-degree lower bound for the independence number; Theorem~\ref{thm:main} then converts it into the corresponding average-degree bound.

\subsection{Graphs excluding fixed cycles and cliques}

We begin with graph classes covered by results of Davies, Kang, Pirot, and Sereni~\cite{DKPS2020}. Their chromatic-number bounds imply maximum-degree lower bounds on the independence number through the elementary inequality
\[
\alpha(G)\ge \frac{|V(G)|}{\chi(G)}.
\]

\begin{theo}[{\cite[Theorem~4]{DKPS2020}}]\label{thm:dkps-cycle}
     For every fixed integer $k\ge 3$, every $C_k$-free graph $G$ of maximum degree $\Delta$ satisfies
     \[\chi(G) \leq (1 + o_{\Delta}(1))\frac{\Delta}{ \log \Delta}.\]
\end{theo}

For a graph $G$, the \emphd{Hall ratio} of $G$ is defined as $$\rho(G)\defeq \max\{|V(F)|/\alpha(F)\,:\, F \textrm{ is a subgraph of } G \}.$$
\begin{theo}[{\cite[Theorem~7]{DKPS2020}}]\label{thm:dkps-hallratio}
    There is a monotone increasing function $K \colon [1, \infty) \to [1, \infty)$ satisfying $K(1) = 1$ and $K(\rho) = (1 + o_{\rho}(1)) \log \rho$ such that the following holds. For every $\rho \ge 1$ and every graph $G$ of maximum degree $\Delta$ in which the neighborhood of every vertex induces a subgraph of Hall ratio at most $\rho$, we have \[\chi(G) \le (K(\rho) + o_{\Delta}(1))\Delta / \log \Delta.\]
\end{theo}

\begin{theo}[{\cite[Theorem~8]{DKPS2020}}]\label{thm:dkps-clique}
    For any graph $G$ of clique number $\omega$ and maximum degree $\Delta$, we have
    \[
    \chi(G) \leq (1 + o(1)) \min \left\{ (\omega - 2) \frac{\Delta \log \log \Delta}{\log \Delta}, 5\Delta \sqrt{\frac{\log(\omega - 1)}{\log \Delta}} \right\}.
    \]
\end{theo}

Applying Theorem~\ref{thm:main} to these inputs gives the following lower bounds for independence number of corresponding graph classes.

\begin{corl}\label{cor:cyclefree}
    Fix $k\geq 3$. Every $C_k$-free graph $G$ of average degree $d$ satisfies
    $$\alpha(G)\geq (1-o_d(1))\dfrac{\log d}{d}|V(G)|.$$
\end{corl}

\begin{proof}
By \Cref{thm:dkps-cycle}, every $C_k$-free graph of maximum degree $\Delta$ satisfies
\[
\alpha(G)\ge \frac{|V(G)|}{\chi(G)}\ge (1-o(1))\frac{\log\Delta}{\Delta}|V(G)|.
\]
The class of $C_k$-free graphs is hereditary, and $f(x)=\log x$ is nearly logarithmic. Apply \Cref{thm:main} with $r=1$.
\end{proof}

\begin{corl}\label{cor:Hallratio}
    There is a monotone increasing function $K \colon [1, \infty) \to [1, \infty)$ satisfying $K(1) = 1$ and $K(\rho) = (1 + o(1)) \log \rho$ as $\rho \to \infty$ such that the following holds. For every $\rho \ge 1$ and every graph $G$ of average degree $d$ in which the neighborhood of every vertex $u \in V(G)$ induces a subgraph of Hall ratio at most $\rho$,
\[
\alpha(G)\ge \left(\frac1{K(\rho)}-o_d(1)\right)\frac{\log d}{d}|V(G)|.
\]
\end{corl}

\begin{proof}
By \Cref{thm:dkps-hallratio}, every graph in this class with maximum degree $\Delta$ satisfies
\[
\alpha(G)\ge \frac{|V(G)|}{\chi(G)}\ge \left(\frac1{K(\rho)}-o_\Delta(1)\right)\frac{\log\Delta}{\Delta}|V(G)|.
\]
The class is hereditary, since the neighborhood of a vertex in an induced subgraph is an induced subgraph of its neighborhood in the original graph. The function $f(x)=\log x/K(\rho)$ is nearly logarithmic. Apply \Cref{thm:main} with $r=1$.
\end{proof}

\begin{corl}\label{cor:cliquefree}
     Fix $k\geq 3$. Every $K_{k+1}$-free $G$ of average degree $d$ satisfies
    \[
    \alpha(G) \geq (1 -o_d(1)) \max \left\{  \frac{\log d}{(k - 2)d \log \log d},  \frac{1}{5d}\sqrt{\frac{\log d}{\log(k - 1)}} \right\}.
    \]
\end{corl}
\begin{proof}
If $G$ is $K_{k+1}$-free, then $\omega(G)\le k$. Hence \Cref{thm:dkps-clique} yields
\[
\alpha(G)\ge \frac{|V(G)|}{\chi(G)}\ge (1-o(1))|V(G)|\max\left\{\frac{\log\Delta}{(k-2)\Delta\log\log\Delta},\,\frac{1}{5\Delta}\sqrt{\frac{\log\Delta}{\log(k-1)}}\right\}.
\]
For fixed $k$, the first term eventually dominates the second. Thus the maximum-degree input holds with
\[
f(x)=\frac{\log x}{(k-2)\log\log x},
\]
which is nearly logarithmic. The class of $K_{k+1}$-free graphs is hereditary. Apply \Cref{thm:main} with $r=1$.
\end{proof}

\subsection{Locally colorable graphs and locally sparse hypergraphs}

We next use fractional-coloring bounds of Dhawan~\cite{Dhawan2026entropy}. These bounds apply to degenerate graphs and hypergraphs; for our purposes, a maximum-degree bound is obtained by noting that maximum degree at most $\Delta$ implies $\Delta$-degeneracy, and that
\[
\alpha(H)\ge \frac{|V(H)|}{\chi_f(H)}.
\]
A graph is \emphd{locally $q$-colorable} if the subgraph induced by the neighborhood of every vertex is $q$-colorable. For hypergraphs, we use the term \emphd{locally sparse} to mean girth at least $4$, that is, having no $2$- or $3$-cycles.

\begin{theo}[{\cite[Theorem~1.2]{Dhawan2026entropy}}]\label{thm:dhawan-graph}
    For all $\varepsilon > 0$ and $q \in \mathbb{N}$ there exists $d_0 \in \mathbb{N}$ such that the following holds for $d \geq d_0$ and $n \in \mathbb{N}$. Let $G$ be an $n$-vertex $d$-degenerate locally $q$-colorable graph. Then
    \[
    \chi_f(G) \leq (8 + \varepsilon) \frac{d \log(2q)}{\log d}.
    \]
\end{theo}

\begin{theo}[{\cite[Theorem~1.3]{Dhawan2026entropy}}]\label{thm:dhawan-hyper}
    For all $\varepsilon > 0$ and $r \in \N$ there exists $d_0 \in \mathbb{N}$ such that the following holds for $d \geq d_0$ and $n \in \mathbb{N}$. Let $H$ be an $n$-vertex $(r+1)$-uniform $d$-degenerate locally sparse hypergraph. Then
    \[
    \chi_f(H) \leq (1 + \varepsilon) \left( \frac{r+1}{r} \right) \left(  \frac{(r+1)r^2 d}{\log d} \right)^{\frac{1}{r}}.
    \]
\end{theo}

Applying Theorem~\ref{thm:main} to the two preceding fractional-coloring bounds gives the following lower bound for independence number of locally $q$-colorable graphs and locally sparse uniform hypergraphs.

\begin{corl}\label{cor:locally-colorable}
    Fix $q\in\N$. Every locally $q$-colorable graph $G$ of average degree $d$ satisfies
    $$\alpha(G)\geq (1-o_d(1))\dfrac{\log d}{8d\log (2q)}|V(G)|.$$
\end{corl}

\begin{proof}
A graph of maximum degree $\Delta$ is $\Delta$-degenerate. Hence \Cref{thm:dhawan-graph} implies that every locally $q$-colorable graph of maximum degree $\Delta$ satisfies
\[
\alpha(G)\ge \frac{|V(G)|}{\chi_f(G)}\ge \left(\frac{1}{8\log(2q)}-o(1)\right)\frac{\log\Delta}{\Delta}|V(G)|.
\]
The function $f(x)=\log x/(8\log(2q))$ is nearly logarithmic, and the class of locally $q$-colorable graphs is hereditary. Apply \Cref{thm:main} with $r=1$.
\end{proof}

\begin{corl}\label{cor:locally-sparse}
    Fix $r\ge 1$. Every locally sparse $(r+1)$-uniform hypergraph $H$ of average degree $d$ satisfies
    $$\alpha(H)\geq (1-o_d(1)) \dfrac{r^{\frac{r-2}{r}}}{(r+1)^{\frac{r+1}{r}} } \left(\dfrac{\log d}{d}\right)^{\frac{1}{r}}|V(H)|.$$
\end{corl}
\begin{proof}
An $(r+1)$-uniform hypergraph of maximum degree $\Delta$ is $\Delta$-degenerate. Hence
Theorem~\ref{thm:dhawan-hyper} implies that every $(r+1)$-uniform hypergraph of girth at least
$4$ and maximum degree $\Delta$ satisfies
\[
\chi_f(H)\le
(1+o_\Delta(1))\left(\frac{r+1}{r}\right)
\left(\frac{(r+1)r^2\Delta}{\log\Delta}\right)^{1/r}.
\]
Therefore
\[
\alpha(H)\ge \frac{|V(H)|}{\chi_f(H)}
\ge
(1-o_\Delta(1))
\left(\frac{r}{r+1}\right)
\left(\frac{\log\Delta}{(r+1)r^2\Delta}\right)^{1/r}
|V(H)|.
\]
The class of $(r+1)$-uniform hypergraphs of girth at least $4$ is hereditary. Moreover,
\[
f(x)=
\left(\frac{r}{r+1}\right)
\left(\frac{\log x}{(r+1)r^2}\right)^{1/r}
\]
is nearly logarithmic. Applying Theorem~\ref{thm:main} gives
    $$\alpha(H)\geq (1-o_d(1)) \dfrac{r^{\frac{r-2}{r}}}{(r+1)^{\frac{r+1}{r}} } \left(\dfrac{\log d}{d}\right)^{\frac{1}{r}}|V(H)|.$$
\end{proof}

\section{Proof of \Cref{thm:main}}\label{sec:proof}
The proof is a cleaning argument. We first choose parameters so that the nearly logarithmic function changes little on the degree ranges produced by the cleaning process. We then repeatedly delete vertices whose degree is larger than a fixed multiple of the current average degree, and analyze the three possible stopping conditions.
\begin{proof}[Proof of \Cref{thm:main}]
Let $\varepsilon_0>0$ be fixed. Choose $\eps\in(0,1/2)$ sufficiently small so that
\[
(1-\eps)^2(1+\eps)^{-1/r}\ge 1-\varepsilon_0.
\tag{3.1}\label{eq:eps-choice}
\]

We first choose the auxiliary parameters. Applying Definition~\ref{def:nearly-log} with
$m=0$ and error parameter $\eps$, we obtain $M_0>0$ and $\lambda_0>1$ such that for every
$x\ge M_0$,
\[
\inf_{y\in[x, \lambda_0x]}\frac{f(y)}{f(x)}\ge 1-\eps.
\tag{3.2}\label{eq:right-window}
\]
Choose
\[
0<\eta<\min\left\{\eps,\lambda_0-1,\frac{r}{r+1}\right\}.
\]
Set
\[
    c\defeq \frac{(r+1)\eta}{r},\qquad
    a\defeq \dfrac{2}{c}+1,\qquad
    b\defeq 4r+\dfrac{2r}{c}.
    \] 
    Then 
\[
0<c<1,\qquad ac-c=2,\qquad b/r-a=3.
\tag{P1}\label{eq:P1}
\]

Next, applying Definition~\ref{def:nearly-log} with $m=b$ and error parameter $\eps$, we obtain
$M_1>0$ and $\lambda_1>1$ such that for every $x\ge M_1$,
\[
\inf_{y\in[xf(x)^{-b}, \lambda_1x]}\frac{f(y)}{f(x)}\ge 1-\eps.
\]
In particular,
\[
\inf_{y\in[xf(x)^{-b}, x]}\frac{f(y)}{f(x)}\ge 1-\eps.
\tag{3.3}\label{eq:left-window-short}
\]

Since $f(x)=x^{o(1)}$ and $\liminf_{x\to\infty}f(x)>2$, after increasing
$M_0$ and $M_1$ if necessary, we may assume that for every
\[
x\ge M_\ast\defeq \max\{M_0,M_1\},
\]
we have
\[
2<f(x)<x^{1/b}\le x^{1/r}
\qquad\text{and}\qquad
4f(x)^{a+1}<x^{1/r}.
\tag{P2}\label{eq:P2}
\]
Moreover, since $1+\eta\le \lambda_0$, \eqref{eq:right-window} gives
\[
\inf_{y\in[x,(1+\eta)x]}\frac{f(y)}{f(x)}\ge 1-\eps
\tag{P3a}\label{eq:P3a}
\]
for every $x\ge M_\ast$, while \eqref{eq:left-window-short} gives
\[
\inf_{y\in[xf(x)^{-b},x]}\frac{f(y)}{f(x)}\ge 1-\eps
\tag{P3b}\label{eq:P3b}
\]
for every $x\ge M_\ast$.

By the maximum-degree hypothesis of the theorem, there exists $\Delta_0=\Delta_0(\eps)$ such that
every hypergraph $F\in\G$ with $\Delta(F)\ge\Delta_0$ satisfies
\[
\alpha(F)\ge
(1-\eps)\frac{f(\Delta(F))}{\Delta(F)^{1/r}}|V(F)|.
\tag{3.4}\label{eq:max-input}
\]
Since $f(x)=x^{o(1)}$, we have
\[
\frac{x}{f(x)^b}\to\infty.
\]
Therefore, after increasing $M_\ast$ if necessary, we may also assume that
\[
\frac{x}{f(x)^b}\ge\Delta_0
\qquad\text{for every }x\ge M_\ast.
\tag{P4}\label{eq:P4}
\]
Set $M:=M_\ast$.    

    Let $H\in\G$ be an $(r+1)$-uniform hypergraph with $n\defeq |V(H)|$ and average degree $d\defeq d(H)>M$.
    For all $i\geq 0$, set
    \[
    U_0\defeq V(H),\qquad H_i\defeq H[U_i],
    \]
    and write
    \[
    N_i\defeq |U_i|,\qquad D_i\defeq d(H_i),\qquad \Delta_i\defeq \Delta(H_i).
    \]
    At stage $i$, stop if one of the following three conditions holds:
\begin{enumerate}[label=\ep{\normalfont{}S\arabic*}]
    \item \label{item:S1} $D_i<\dfrac{d}{f(d)^b}$,
    \item \label{item:S2}  $\dfrac{n}{N_i}>f(d)^a$,
    \item \label{item:S3} $\Delta_i\le (1+\eta)D_i$.
\end{enumerate}

    If none of \ref{item:S1}--\ref{item:S3} holds, then there exists a vertex $v_i\in U_i$ such that
    \[
    \deg_{H_i}(v_i)>(1+\eta)D_i.
    \]
    Define
    \[
   U_{i+1}\defeq U_i\setminus\{v_i\},\qquad H_{i+1}\defeq H[U_{i+1}].
    \]
    Since $\G$ is hereditary, $H_{i+1}\in\G$.  The procedure terminates after finitely many steps; let $T$ be the stopping time.

\begin{lemma}\label{lem:QT-gain}
    For $0\leq i\leq T$, define
    \[
    Q_i\defeq \frac{N_i}{D_i^{1/r}}.
    \]
    with the convention that $Q_i=+\infty$ if $D_i=0$. Then \[
    Q_T\ge\left(\frac{n+1}{N_T+1}\right)^c\frac{n}{d^{1/r}}\geq \frac{n}{d^{1/r}}.\]
\end{lemma}

\begin{proof}
    We claim that whenever stage $i$ is a genuine cleaning step, one has
    \[
    Q_{i+1}\ge \left(1+\frac{c}{N_i}\right)Q_i.
    \tag{3.5}\label{eq:single-step-gain}
    \]
If $D_{i+1}=0$, then the claim is immediate from the convention above. Thus assume $D_{i+1}>0$.

    Since $H_i$ is $(r+1)$-uniform,
    \[
    e(H_i)=\frac{N_iD_i}{r+1}.
    \]
The hypergraph $H_{i+1}$ is obtained from $H_i$ by deleting $v_i$ and then taking the induced subhypergraph on the remaining vertices. Hence all hyperedges containing $v_i$ disappear, and
    \[
    e(H_{i+1})=e(H_i)-\deg_{H_i}(v_i)<\frac{N_iD_i}{r+1}-(1+\eta)D_i.
\]
Therefore
\[
D_{i+1}=\frac{(r+1)e(H_{i+1})}{N_i-1}<\frac{N_i-(r+1)(1+\eta)}{N_i-1}\,D_i.
\]
It follows that
\[
Q_{i+1}=\frac{N_i-1}{D_{i+1}^{1/r}}>\frac{N_i-1}{D_i^{1/r}}\left(1-\frac{r+(r+1)\eta}{N_i-1}\right)^{-1/r}.
\]
The term in parentheses is positive, since
\[
\deg_{H_i}(v_i)\le e(H_i)=\frac{N_iD_i}{r+1}
\]
and $\deg_{H_i}(v_i)>(1+\eta)D_i$ imply $N_i>(r+1)(1+\eta)$.
Using $(1-x)^{-1/r}\ge 1+\frac{x}{r}$ for $0\le x<1$,
we obtain
\[
Q_{i+1}\ge\frac{N_i-1}{D_i^{1/r}}\left(1+\frac{r+(r+1)\eta}{r(N_i-1)}\right)=\left(1+\frac{(r+1)\eta}{rN_i}\right)\frac{N_i}{D_i^{1/r}}=\left(1+\frac{c}{N_i}\right)Q_i.
\]
This proves \eqref{eq:single-step-gain}.

Since each cleaning step removes exactly one vertex, the values $N_i$ encountered before stopping are
\[
n,n-1,n-2,\dots,N_T+1.
\]
Multiplying \eqref{eq:single-step-gain} over all cleaning steps, we obtain
\[
Q_T\ge\prod_{j=N_T+1}^{n}\Bigl(1+\frac{c}{j}\Bigr)\frac{n}{d^{1/r}}.
\]
Because $0<c\le 1$, we have
$1+c/j\ge (1+1/j)^c$,
and therefore
\[
Q_T\ge\prod_{j=N_T+1}^{n}\left(1+\frac{1}{j}\right)^c\frac{n}{d^{1/r}}=\left(\frac{n+1}{N_T+1}\right)^c\frac{n}{d^{1/r}}\ge \frac{n}{d^{1/r}}.\qedhere
\]
\end{proof}

\begin{lemma}\label{lem:crude-bound}
    For any $(r+1)$-uniform hypergraph $F$,
    \[
    \alpha(F)\ge \frac{r}{r+1}|V(F)|\min\{1,d(F)^{-1/r}\}.
    \]
\end{lemma}

\begin{proof}
Let
\[
p\defeq \min\{1,d(F)^{-1/r}\}.
\]
Choose each vertex of $F$ independently with probability $p$, and let $S$ be the resulting random set. Then
\[
\mathbb E|S|=p|V(F)|
\]
and, since $F$ is $(r+1)$-uniform,
\[
\mathbb E e(F[S])=p^{r+1}e(F)=p^{r+1}\frac{|V(F)|d(F)}{r+1}\le \frac{p|V(F)|}{r+1}.
\]
Deleting one vertex from each edge of $F[S]$ leaves an independent set. Hence
\[
\alpha(F)\ge \mathbb E|S|-\mathbb E e(F[S])\ge p|V(F)|-\frac{p|V(F)|}{r+1}=\frac{r}{r+1}p|V(F)|.\qedhere
\]
\end{proof}

We distinguish three cases according to the stopping condition.

\smallskip
\noindent
\emph{Case 1: \ref{item:S1} holds.}
Thus
\[
D_T<\frac{d}{f(d)^b}.
\]
Since $D_0=d$, condition \ref{item:S1} cannot hold at time $0$, so $T\ge 1$. Hence \ref{item:S2} failed at time $T-1$, which means
\[
\frac{n}{N_{T-1}}\le f(d)^a.
\]
Since $N_T+1=N_{T-1}$, it follows that
\[
N_T+1\ge \frac{n}{f(d)^a}.
\]
Also, since $H$ is $(r+1)$-uniform, we have $d\le \binom{n-1}{r}<n^r$. Together with \eqref{eq:P2}, this gives
\[
n>d^{1/r}>4f(d)^{a+1}>2f(d)^a.
\]
Consequently,
\[
N_T\ge \frac{n}{2f(d)^a}.
\tag{3.6}\label{eq:NT-lower}
\]

If $D_T\ge 1$, then by Lemma \ref{lem:crude-bound},
\[
\alpha(H)\ge \alpha(H_T)\ge \frac{r}{r+1}\frac{N_T}{D_T^{1/r}}\ge\frac{r}{2(r+1)}\frac{n}{f(d)^a}\left(\frac{f(d)^b}{d}\right)^{1/r}=\frac{r}{2(r+1)}f(d)^{b/r-1-a}  \dfrac{f(d)}{d^{1/r}}n.
\]
By \eqref{eq:P1}, $b/r-1-a = 2$. Since $f(d)\geq 2$, we have \[\frac{r}{2(r+1)}f(d)^{b/r-1-a} = \dfrac{r}{2(r+1)}f(d)^2\geq 1 \quad \text{ and hence } \quad\alpha(H)\ge  \dfrac{f(d)}{d^{1/r}}n.\]

If instead $D_T<1$, then by \Cref{lem:crude-bound} and \eqref{eq:NT-lower},
\[
\alpha(H)\ge \alpha(H_T)\ge \frac{r}{r+1}N_T\ge \frac{r}{2(r+1)}\frac{n}{f(d)^a}=\frac{r}{2(r+1)}\dfrac{d^{1/r}}{f(d)^{a+1}} \frac{f(d)}{d^{1/r}}n.
\]
Again by \eqref{eq:P2} with $x=d$,
\[
4f(d)^{a+1}<d^{1/r}.
\]
Therefore
\[
\frac{r}{2(r+1)}
\frac{d^{1/r}}{f(d)^{a+1}}
>
\frac{2r}{r+1}
\ge1,
\]
and so
\[
\alpha(H)\ge \frac{f(d)}{d^{1/r}}n.
\]
\smallskip

\noindent
\emph{Case 2: \ref{item:S2} holds and \ref{item:S1} does not.}
Then
\[
\frac{n}{N_T}>f(d)^a \qquad\text{and}\qquad D_T\ge \frac{d}{f(d)^b}>1.
\]
By Lemma \ref{lem:QT-gain} and Lemma \ref{lem:crude-bound}, we get
\[
\alpha(H)\ge \alpha(H_T)\ge \frac{r}{r+1}\frac{N_T}{D_T^{1/r}}\ge\frac{r}{r+1}\left(\frac{n+1}{N_T+1}\right)^c\frac{n}{d^{1/r}}.
\]
Since
\[
\frac{n+1}{N_T+1}\ge \frac{n}{2N_T}>\frac12f(d)^a,
\]
we get
\[
\alpha(H)\ge\frac{r}{r+1}\,2^{-c}f(d)^{ac}\frac{n}{d^{1/r}}=\frac{r}{r+1}\,2^{-c}f(d)^{ac-1}\frac{f(d)}{d^{1/r}}n.
\]
By \eqref{eq:P1}, $ac-1 = c+1$. Again, since $f(d) > 2$, we have
$$ \frac{r}{r+1}\,2^{-c}f(d)^{ac-1} = \frac{r}{r+1} f(d) \,2^{-c}f(d)^{c}\geq 1\quad \text{ and hence }\quad \alpha(H)\ge \frac{f(d)}{d^{1/r}}n.$$
\smallskip

\noindent
\emph{Case 3: \ref{item:S3} holds and \ref{item:S1}, \ref{item:S2} do not.}
    Then
    \[
    \dfrac{d}{f(d)^b}\leq D_T\leq \Delta_T\le (1+\eta)D_T\leq (1+\eta) d.
    \]
    Since $\G$ is hereditary, $H_T\in \G$. Moreover, by \eqref{eq:P4},
\[
\Delta_T\ge D_T\ge \frac{d}{f(d)^b}\ge \Delta_0.
\]
Applying \eqref{eq:max-input} to $H_T$ gives
\[
\alpha(H_T)
\ge
(1-\eps)\frac{f(\Delta_T)}{\Delta_T^{1/r}}N_T.
\]
Now
\[
\Delta_T\in\left[\frac{d}{f(d)^b},(1+\eta)d\right].
\]
If $\Delta_T\le d$, then \eqref{eq:P3b} gives
\[
f(\Delta_T)\ge (1-\eps)f(d).
\]
If $\Delta_T\ge d$, then \eqref{eq:P3a} gives the same inequality, since $\Delta_T\le(1+\eta)d$. Thus, in all cases,
\[
f(\Delta_T)\ge (1-\eps)f(d).
\]
Also,
\[
\Delta_T^{1/r}\le ((1+\eta)D_T)^{1/r}.
\]
Therefore
\[
\alpha(H_T)
\ge
(1-\eps)^2(1+\eta)^{-1/r}
f(d)\frac{N_T}{D_T^{1/r}}.
\]
    By Lemma \ref{lem:QT-gain}, $$\dfrac{N_T}{D_T^{1/r}}\geq \dfrac{n}{d^{1/r}}.$$
    Hence
\[
\alpha(H)
\ge
\alpha(H_T)
\ge
(1-\eps)^2(1+\eta)^{-1/r}
\frac{f(d)}{d^{1/r}}n.
\]
Since $\eta<\eps$,
\[
(1+\eta)^{-1/r}\ge(1+\eps)^{-1/r}.
\]
By \eqref{eq:eps-choice},
\[
\alpha(H)\ge
(1-\varepsilon_0)\frac{f(d)}{d^{1/r}}n.
\]
Since $\varepsilon_0>0$ was arbitrary, the theorem follows.
\end{proof}

\begin{remk}[Weighted variant]
The same cleaning argument gives a weighted version of Theorem~\ref{thm:main}.
Let $H$ be an $(r+1)$-uniform hypergraph and let
$w:V(H)\to\mathbb R_{>0}$. Define
\[
\lambda_H^w(v):=
\frac{1}{w(v)^r}\sum_{e\ni v}\prod_{u\in e\setminus\{v\}}w(u),
\qquad
\Delta_w(H):=\max_{v\in V(H)}\lambda_H^w(v),
\]
and
\[
\mu_w(U):=\sum_{u\in U}w(u)^{r+1},
\qquad
d_w(H):=
\frac{(r+1)\sum_{e\in E(H)}\prod_{u\in e}w(u)}
{\mu_w(V(H))}.
\]
Also let
\[
\alpha_w^\ast(H):=
\max\left\{\sum_{v\in I}w(v)^{r+1}: I\subseteq V(H)\text{ is independent}\right\}.
\]
If the maximum-degree hypothesis in Theorem~\ref{thm:main} holds in the corresponding weighted
form with $\Delta_w$ and $\alpha_w^\ast$, then the same conclusion holds with $d_w$ and
$\mu_w(V(H))$ in place of $d$ and $|V(H)|$. The proof is identical after replacing $N_i$ by
$\mu_w(U_i)$ and the edge count by the weighted edge-mass.
\end{remk}

\section*{Acknowledgment}
We thank Abhishek Dhawan for helpful comments on an earlier version of this manuscript.

\printbibliography

@article {Ajtai1980Ramsy,
    AUTHOR = {Ajtai, Mikl\'os and Koml\'os, J\'anos and Szemer\'edi, Endre},
     TITLE = {A note on {R}amsey numbers},
   JOURNAL = {J. Combin. Theory Ser. A},
  FJOURNAL = {Journal of Combinatorial Theory. Series A},
    VOLUME = {29},
      YEAR = {1980},
    NUMBER = {3},
     PAGES = {354--360},
      ISSN = {0097-3165,1096-0899},
   MRCLASS = {05C55 (05C35)},
  MRNUMBER = {600598},
MRREVIEWER = {J.\ E.\ Graver},
       DOI = {10.1016/0097-3165(80)90030-8},
       URL = {https://doi.org/10.1016/0097-3165(80)90030-8},
}

@article {Shearer1983ind1,
    AUTHOR = {Shearer, James B.},
     TITLE = {A note on the independence number of triangle-free graphs},
   JOURNAL = {Discrete Math.},
  FJOURNAL = {Discrete Mathematics},
    VOLUME = {46},
      YEAR = {1983},
    NUMBER = {1},
     PAGES = {83--87},
      ISSN = {0012-365X,1872-681X},
   MRCLASS = {05C99},
  MRNUMBER = {708165},
MRREVIEWER = {Linda\ Lesniak},
       DOI = {10.1016/0012-365X(83)90273-X},
       URL = {https://doi.org/10.1016/0012-365X(83)90273-X},
}

@article {Anton2023ForbidBipartite,
    AUTHOR = {Anderson, James and Bernshteyn, Anton and Dhawan, Abhishek},
     TITLE = {Colouring graphs with forbidden bipartite subgraphs},
   JOURNAL = {Combin. Probab. Comput.},
  FJOURNAL = {Combinatorics, Probability and Computing},
    VOLUME = {32},
      YEAR = {2023},
    NUMBER = {1},
     PAGES = {45--67},
      ISSN = {0963-5483,1469-2163},
   MRCLASS = {05C15},
  MRNUMBER = {4523856},
MRREVIEWER = {Deming\ Li},
       DOI = {10.1017/s0963548322000104},
       URL = {https://doi.org/10.1017/s0963548322000104},
}

@article {Spencer1982hyper,
    AUTHOR = {Ajtai, M. and Koml\'os, J. and Pintz, J. and Spencer, J. and
              Szemer\'edi, E.},
     TITLE = {Extremal uncrowded hypergraphs},
   JOURNAL = {J. Combin. Theory Ser. A},
  FJOURNAL = {Journal of Combinatorial Theory. Series A},
    VOLUME = {32},
      YEAR = {1982},
    NUMBER = {3},
     PAGES = {321--335},
      ISSN = {0097-3165,1096-0899},
   MRCLASS = {05C65},
  MRNUMBER = {657047},
MRREVIEWER = {F.\ Sterboul},
       DOI = {10.1016/0097-3165(82)90049-8},
       URL = {https://doi.org/10.1016/0097-3165(82)90049-8},
}

@inproceedings {Rodl1995hyper,
    AUTHOR = {Duke, Richard A. and Lefmann, Hanno and R\"odl, Vojt\v ech},
     TITLE = {On uncrowded hypergraphs},
 BOOKTITLE = {Proceedings of the {S}ixth {I}nternational {S}eminar on
              {R}andom {G}raphs and {P}robabilistic {M}ethods in
              {C}ombinatorics and {C}omputer {S}cience, ``{R}andom {G}raphs
              '93'' ({P}ozna\'n, 1993)},
   JOURNAL = {Random Structures Algorithms},
  FJOURNAL = {Random Structures \& Algorithms},
    VOLUME = {6},
      YEAR = {1995},
    NUMBER = {2-3},
     PAGES = {209--212},
      ISSN = {1042-9832,1098-2418},
   MRCLASS = {05C65 (05C80)},
  MRNUMBER = {1370956},
       DOI = {10.1002/rsa.3240060208},
       URL = {https://doi.org/10.1002/rsa.3240060208},
}

@article {Verstraete2026hyper,
    AUTHOR = {Verstraete, Jacques and Wilson, Chase},
     TITLE = {Independent sets in hypergraphs},
   JOURNAL = {Random Structures Algorithms},
  FJOURNAL = {Random Structures \& Algorithms},
    VOLUME = {68},
      YEAR = {2026},
    NUMBER = {1},
     PAGES = {Paper No. e70047, 9},
      ISSN = {1042-9832,1098-2418},
   MRCLASS = {05C69 (05C65)},
  MRNUMBER = {5022514},
       DOI = {10.1002/rsa.70047},
       URL = {https://doi.org/10.1002/rsa.70047},
}

@unpublished{Abhi2025Kttt,
	author = {Abhishek Dhawan and Oliver Janzer and Abhishek Methuku},
	title = {Independent sets and colorings of $K_{t,t,t}$-free graphs},
	howpublished = {\url{https://arxiv.org/abs/2511.17191} (preprint)},
	date = {2025},
}

@unpublished{Dhawan2026entropy,
	author = {Abhishek Dhawan},
	title = {Fractional coloring via entropy},
	howpublished = {\url{https://arxiv.org/abs/2603.17730} (preprint)},
	date = {2026},
}

@unpublished{DKPS2020,
	 title={Graph structure via local occupancy},
  author={Davies, Ewan and Kang, Ross J and Pirot, Fran{\c{c}}ois and Sereni, Jean-S{\'e}bastien},
	howpublished = {\url{https://arxiv.org/abs/2003.14361} (preprint)},
	date = {2020},
}

@article {alon1999coloring,
    AUTHOR = {Alon, Noga and Krivelevich, Michael and Sudakov, Benny},
     TITLE = {Coloring graphs with sparse neighborhoods},
   JOURNAL = {J. Combin. Theory Ser. B},
  FJOURNAL = {Journal of Combinatorial Theory. Series B},
    VOLUME = {77},
      YEAR = {1999},
    NUMBER = {1},
     PAGES = {73--82},
      ISSN = {0095-8956,1096-0902},
   MRCLASS = {05C15},
  MRNUMBER = {1710532},
MRREVIEWER = {David\ E.\ Woolbright},
       DOI = {10.1006/jctb.1999.1910},
       URL = {https://doi.org/10.1006/jctb.1999.1910},
}
\end{document}